\documentclass[12pt,sort&compress]{elsarticle}
\usepackage{amsmath,amsfonts,url}
\usepackage{microtype}
\usepackage{tikz}
\usepackage[inline]{enumitem}

\newtheorem{theorem}{Theorem}
\newtheorem{proposition}{Proposition}
\newproof{proof}{Proof}
\newdefinition{definition}{Definition}
\newdefinition{algorithm}{Procedure}
\newcommand{\cross}{\operatorname{cr}}
\newcommand{\nest}{\operatorname{ne}}
\newcommand{\figurefontsize}{\footnotesize}
\tikzstyle{pnt}=[circle,fill,inner sep=1pt]
\tikzstyle{cnt}=[dashed, very thick]
\tikzstyle{loop above}=[above,out=135,in=45,loop]
\newcommand{\Arc}{\operatorname{Arc}}

\begin{document}

\begin{frontmatter}
\title{A Bijection for Crossings and Nestings}
\author{Lily Yen}
\address{Dept.\ of Math.\ \& Stats., Capilano University, North Vancouver, B.C., Canada;\\{\normalfont also:} Dept.\ of Math., Simon Fraser University, Burnaby, B.C., Canada}
\begin{keyword}
\MSC[2010] 05A19
\end{keyword}

\begin{abstract}
For a subclass of matchings, set partitions, and  permutations, we describe a direct bijection involving only arc annotated diagrams that not only interchanges maximum nesting and crossing numbers, but also all refinements of crossing and nesting numbers.
Furthermore, we show that the bijection cannot be extended to a larger class of arc annotated diagrams while retaining a global structure. We apply the bijection to a similar subclass of coloured  matchings, set partitions, and permutations.
\end{abstract}
\end{frontmatter}
\thispagestyle{empty}

\section{Introduction}\label{sec:introduction}

A perfect matching on $[2n] = \{ 1, 2, 3, \dots, 2n\}$ is a partition of $[2n]$ into only $2$-element subsets. We can list its $n$ blocks as $\{ (i_1, j_1), (i_2, j_2), \dots, (i_n, j_n)\}$ where $i_r < j_r$ for $1 \le r \le n$. Pictorially, a matching can be drawn as an \emph{arc annotated diagram} where the elements $[2n]$ are drawn as vertices on a horizontal line increasingly labelled and elements of the same block are joined with an arc.
An example of a matching thus represented is shown in Figure~\ref{fig:matching}.

\begin{figure}[hbp]
\figurefontsize\centering
\begin{tikzpicture}
   \foreach \i in {1,...,10} {
      \node[pnt] (a\i) at (\i, 0) {};
      \node[right] at (a\i.east) {$\i$};
   }
   \draw[bend left=30] (a1) to (a9);
   \draw[bend left=45] (a2) to (a5);
   \draw[bend left=45] (a3) to (a6);
   \draw[bend left=45] (a4) to (a7);
   \draw[bend left=45] (a8) to (a10);
\end{tikzpicture}
\caption{A Matching of the set $\{1, 2, \dots, 10\}$}
\label{fig:matching}
\end{figure}
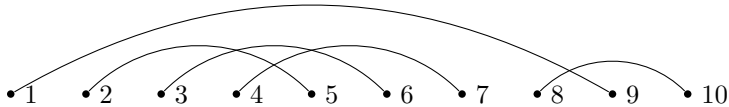

We say that $2$ blocks $\{ (i_r, j_r), (i_s, j_s) \}$ in the matching (or equivalently $2$ arcs in the arc annotated diagram representing the matching) form a \emph{crossing} if $i_r < i_s < j_r < j_s$, that is, their arcs cross. In Figure~\ref{fig:matching}, one notes $4$ crossings, one of which is $\{ (2,5), (3,6)\}$. Analogously, we say that $2$ blocks 
$\{ (i_r, j_r), (i_s, j_s) \}$ form a \emph{nesting} if $i_r < i_s < j_s < j_r$, or the arcs nest. Figure~\ref{fig:matching} shows $3$ nestings, one of which is $\{ (3,6), (1, 9)\}$.

It is well known that the number of perfect matchings on $[2n]$ without any crossing is counted by the $n$-th Catalan number, $c_n = \binom{2n}{n}/(n+1)$ for $n \ge 1$, and so is the number of perfect matchings on $[2n]$ without any nesting. Thus began the exploration of symmetric joint distribution between crossings and nestings for a large class of combinatorial objects. For instance, for the classical reflection groups, Athanasiadis for type $A$~\cite{Atha98} and Fink and Giraldo for the other types~\cite{FiGi10} gave bijective proofs between noncrossing and nonnesting set partitions preserving block sizes. Furthermore, Rubey and Stump~\cite{RuSt10} exhibited bijections on various classes of set partitions of classical types that preserve openers and closers. Parallel to the development of bijective proofs for nonnesting and noncrossing set partitions of classical refection groups is the study of joint distribution between crossing and nesting statistics of two edges in matchings by Klazar~\cite{Klazar06}, set partitions by Poznanovi{\'c} and Yan~\cite{PoYa09} and both by Kasraoui and Zeng~\cite{KaZe06}. Translating arc diagrams of matchings (resp.\ set partitions) to Dyck paths (resp.\ bi-coloured Motzkin paths) combined with tunnels and Charlier diagrams, Klazar, Poznanovi{\'c}, and Yan obtained bijective proofs and generating series as continued fractions. 

When the concept of noncrossing (resp.\ nonnesting)  is extended to $k$-crossing
shown in Figure~\ref{fig:kcrossing} (resp.\ $k$-nesting  shown in Figure~\ref{fig:knesting}), combinatorialists since the 1950's  have found that symmetric joint distribution continues to hold for  $k = 2$ and $3$.
(Formal definitions will be given in Section~\ref{sec:matchings}.) 
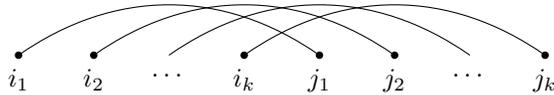
\begin{figure}
\figurefontsize\centering
\begin{tikzpicture}[bend left=35]
   \node[pnt,label=below:$i_1$] at (1,0) {};
   \node[pnt,label=below:$i_2$] at (2,0) {};
   \node[label=below:$\dots$] at (3,0) {};
   \node[pnt,label=below:$i_k$] at (4,0) {};
   \node[pnt,label=below:$j_1$] at (5,0) {};
   \node[pnt,label=below:$j_2$] at (6,0) {};
   \node[label=below:$\dots$] at (7,0) {};
   \node[pnt,label=below:$j_k$] at (8,0) {};
   \draw (1,0) to (5,0);
   \draw (2,0) to (6,0);
   \draw (3,0) to (7,0);
   \draw (4,0) to (8,0);	
\end{tikzpicture}
\caption{The arc diagram of a $k$-crossing}
\label{fig:kcrossing}
\end{figure}

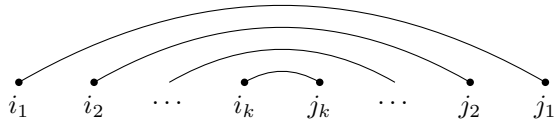
\begin{figure}
\figurefontsize\centering
\begin{tikzpicture}[bend left=30]
   \node[pnt,label=below:$i_1$] at (1,0) {};
   \node[pnt,label=below:$i_2$] at (2,0) {};
   \node[label=below:$\dots$] at (3,0) {};
   \node[pnt,label=below:$i_k$] at (4,0) {};
   \node[pnt,label=below:$j_k$] at (5,0) {};
   \node[label=below:$\dots$] at (6,0) {};
   \node[pnt,label=below:$j_2$] at (7,0) {};
   \node[pnt,label=below:$j_1$] at (8,0) {};
   \draw (1,0) to (8,0);
   \draw (2,0) to (7,0);
   \draw (3,0) to (6,0);
   \draw (4,0) to (5,0);
\end{tikzpicture}
\caption{The arc diagram of a $k$-nesting}
\label{fig:knesting}
\end{figure}

We call a matching without any $k$-nesting a \emph{$k$-nonnesting} matching and similarly for a \emph{$k$-noncrossing} matching. Gouyou-Beauschamps~\cite{GB89} first studied the enumeration of $3$-nonnesting matchings and found symmetric joint distribution between $3$-nonnesting matchings and $3$-noncrossing matchings. In addition,  Chen, Deng, Du, Stanley, and Yan~\cite{Chetal07} showed that for any $k$, the number of $k$-noncrossing matchings of $[2n]$ forms a P-recursive sequence, that is, the sequence satisfies a linear recurrence relation with polynomial coefficients. Bousquet-M\'elou~\cite{Bous11} used the method of generating trees to produce a determinantal expression that enumerates involutions with no long descending subsequence which translates to $k$-nonnesting partial matchings (fixed points allowed).

Chen et al.~\cite{Chetal07} introduced the concept of maximal nesting and crossing of a matching by defining $cr(M)$ to be the maximal $i$ such that $M$ has an $i$-crossing and $ne(M)$ to be the maximal $j$ such that $M$ has a $j$-nesting. Denote by $f_n(i,j)$ the number of matchings $M$ on $[2n]$ with $cr(M) = i$ and $ne(M)=j$. Then they showed a bijection via vacillating tableaux that
\[
f_n(i,j) = f_n(j,i).
\]
In particular, the number of matchings with $cr(M) = k$ is equal to the number of matchings with $ne(M)=k$. Their bijection applies to set partitions as well. Krattenthaler~\cite{Kratt06} proved the same result using growth diagrams. Bousquet-M\'elou and Xin~\cite{BoXi06} showed that the sequence counting $3$-noncrossing partitions is P-recursive. They further conjectured that for $k \ge 4$, the sequence counting $k$-noncrossing partitions is not $P$-recursive.

Extending the construction from matchings and set partitions to permutations, Burrill, Mishna, and Post~\cite{BuMiPo10} applied  de Mier's result on embedded graphs~\cite{deMi07} using Krattenthaler's growth diagrams and fillings of Ferrers diagrams~\cite{Kratt06} to establish a symmetric joint distribution property between $k$-crossing and $k$-nesting for permutations. All these bijections involve passing through a sequence  of tableau related objects or lattice paths.  One enticing question for such combinatorial objects is whether a more direct bijection exists (staying within arc annotated diagrams) that interchanges maximal nesting and crossing numbers. For a smaller class of matchings, set partitions, and permutations,  we show a simple bijection which not only interchanges maximal nesting and crossing numbers, but also all refinements of crossing and nesting numbers. Like bijections of~\cite{KaZe06, Chetal07, RuSt10}, our bijection also preserves openers and closers in addition to a global structure.

The first attempt of turning a maximal crossing into a maximal nesting does not work because the example in Figure \ref{fig:notwork} shows that switching a maximal crossing locally does not produce a matching with the correct maximal nesting.
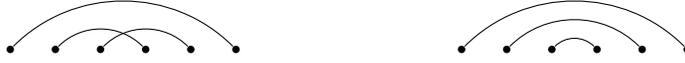
\begin{figure}
\figurefontsize\centering
\begin{tikzpicture}[scale=0.6]
   \node[pnt] at (0,0)(1){};
   \node[pnt] at (1,0)(2){};
   \node[pnt] at (2,0)(3){};
   \node[pnt] at (3,0)(4){};
   \node[pnt] at (4,0)(5){};
   \node[pnt] at (5,0)(6){};
   \node[pnt] at (10,0)(1b){};
   \node[pnt] at (11,0)(2b){};
   \node[pnt] at (12,0)(3b){};
   \node[pnt] at (13,0)(4b){};
   \node[pnt] at (14,0)(5b){};
   \node[pnt] at (15,0)(6b){};
   \draw(1)  to [bend left=45] (6);
   \draw(2)  to [bend left=45] (4);
   \draw(3)  to [bend left=45] (5);
   \draw(1b)  to [bend left=45] (6b);
   \draw(2b)  to [bend left=45] (5b);
   \draw(3b)  to [bend left=45] (4b);
\end{tikzpicture}
\caption{A maximal $2$-crossing turns into a maximal $3$-nesting through local switching}
\label{fig:notwork}
\end{figure}

To simplify the description of the bijection, we first introduce the concept of a decomposable arc annotated diagram. Apply the bijection to each indecomposable structure first, then concatenate the pieces to obtain the result. The plan of the paper is to describe the bijection in detail for perfect matchings, then to use the method of inflation to transform partial matchings and set partitions to perfect matchings where the bijection applies. Finally, we show that permutations can be treated the same way by splitting the arc annotated diagrams to upper and lower arc diagrams. Enumeration of admissible objects for each class follows by first finding the generating series for each indecomposable component, then applying the sequence construction. The functional expressions thus obtained are conjectured to be non-holonomic. We extend the bijection to coloured arc annotated diagrams for matchings, set partitions, and permutations, then close by discussing limitations of this bijection.

\section{Matchings}\label{sec:matchings}

The definitions already introduced in Section \ref{sec:introduction} for perfect matchings are assumed. We define below the concept of  \emph{decomposable} arc diagrams.

\begin{definition}
An arc diagram is \emph{decomposable} if  a vertical line swept from the first vertex of the diagram to the last encounters at least once no intersection with any arc   strictly  between the first and the last vertices.
\end{definition}

When an arc diagram is not decomposable, we call it \emph{indecomposable}. Figure~\ref{fig:matching} is indecomposable. Whenever considering a given arc diagram, we consider each indecomposable interval separately.

The vertices of a matching can be of two types: \emph{openers} and \emph{closers}. Figure~\ref{fig:matching} has $\{1, 2, 3, 4, 8\}$ as the set of openers and $\{5, 6, 7, 9, 10\}$ as its closers; the maximal crossing number for Figure~\ref{fig:matching} is $3$ (from $\{(2,5), (3,6), (4,7) \}$), and the maximal nesting number is $2$.  (Note $3$ $2$-nestings formed by $\{(1, 9), (2, 5) \}$,  $\{(1, 9), (3, 6) \}$,  $\{(1, 9), (4, 7) \}$,  and $ \{(1, 9), (8, 10) \}$  ).  Furthermore, there is a $2$-crossing formed by $\{(1, 9), (8,10)\}$.  The bijection we describe maps Figure~\ref{fig:matching} to  a matching  with a maximal nesting number of $3$, a maximal crossing number of $2$ with $3$ such $2$-crossings, and also a $2$-nesting elsewhere, thus switching all refinements of crossing and nesting numbers.

We restrict matchings on $[2n]$ to a smaller class where the vertices from the left to the right are first $n$ consecutive openers followed by $n$ consecutive closers, or \emph{Type OC}  in short. After the proof of the theorem, we will extend the procedure of the bijection to include matchings of the type shown in Figure~\ref{fig:matching}. We need another definition before the first theorem.

\begin{definition}
Define the label of a matching $\mu$ to be
\[
	L(\mu) = (n_2, n_3, \dots, n_i; c_2, c_3, \dots, c_j)
\]
where $\mu$ contains $n_2$ $2$-nestings, $n_3$ $3$-nestings, \dots, $n_i$ $i$-nestings and $c_2$ $2$-crossings, $c_3$ $3$-crossings, \dots, $c_j$ $j$-crossings.
\end{definition}
For $\mu$ of Figure~\ref{fig:matching}, $L(\mu)= (3; 4,1)$.
\begin{theorem}
\label{thm:matching}
Given a Type OC matching $\mu \in \mathfrak{M}_{2n}$, that is, $\mu$ contains  $n$ consecutive openers followed by $n$ consecutive closers in an indecomposable interval, such that the label of $\mu$ is
\[
	L(\mu) = (n_2, n_3, \dots, n_i; c_2, c_3, \dots, c_j),
\]
then there is an involution mapping $\mu \in \mathfrak{M}_{2n}$ to a Type OC matching, $\nu \in \mathfrak{M}_{2n}$ with
\[
	L(\nu) = ( c_2, c_3, \dots, c_j;n_2, n_3, \dots, n_i).
\]
\end{theorem}

The intuition of the proof comes from a simple reverse relabelling of the closers because a $k$-nesting  and a $k$-crossing are mapped to each other by a simple reverse relabelling of the closers. This idea  may seem extremely na\"\i ve at first, but for this particular subclass of matchings, the idea is sufficient. Formally, a $k$-\emph{crossing} of~$\mu$ is a collection
of~$k$ arcs~$(i_1, j_1)$, $(i_2, j_2)$, \dots, $(i_k, j_k)$ such that
$i_1<i_2< \dots < i_k <j_1 < j_2 < \dots < j_k$ as shown in Figure~\ref{fig:kcrossing}. Similarly, we define a $k$-\emph{nesting} of $\mu$ to be a collection
of $k$ arcs $(i_1, j_1)$, $(i_2, j_2)$, \dots, $(i_k, j_k)$ such that
$i_1<i_2< \dots < i_k <j_k < j_{k-1} < \dots < j_1$ as shown in Figure~\ref{fig:knesting}.

Here, it is clearly seen in Figures~\ref{fig:kcrossing} and~\ref{fig:knesting} that $j_1$, $j_2$, \dots, $j_k$ are labelled in reverse from one diagram to the other. We remark that one can also reverse the labels of the $i$'s (the openers) instead of the $j$'s (the closers) to achieve the same effect depending on the situation.

\begin{proof}[Subsequence approach]
Consider a Type OC matching, $\mu \in \mathfrak{M}_{2n}$. Our goal is to reverse the closer labels of all nestings and crossings. Since  labelling all closers in reverse turns an originally increasing subsequence in the closers to decreasing, all crossings and nestings are switched.
\end{proof}

\begin{proof}[Two-line approach inspired by \cite{MeRo08}]
Pick arbitrarily two o\-pen\-er vertices of $\mu$ , say $a$, $b$, where $a < b$. As the openers are closed, these two arcs either cross or nest.

Represent $\mu$ in a two-line diagram where openers are drawn on the first line increasingly labelled, and closers are drawn on the second line, also increasingly labelled. When $k$ arcs form a $k$-crossing in the standard representation, the corresponding $k$ line segments in the two-line diagram are all non-intersecting. Figure~\ref{fig:2linecross} shows a $k$-crossing from Figure~\ref{fig:kcrossing} drawn in a two-line diagram.
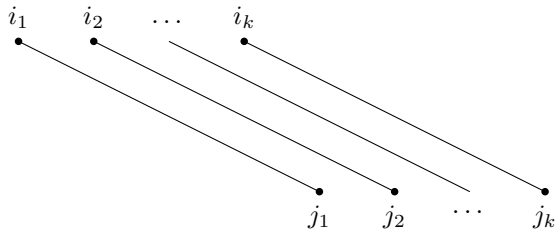
\begin{figure}
\figurefontsize\centering
\begin{tikzpicture}
   \node[pnt,label=above:$i_1$] at (1,0) {};
   \node[pnt,label=above:$i_2$] at (2,0) {};
   \node[label=above:$\dots$] at (3,0) {};
   \node[pnt,label=above:$i_k$] at (4,0) {};
   \node[pnt,label=below:$j_k$] at (8,-2) {};
   \node[label=below:$\dots$] at (7,-2) {};
   \node[pnt,label=below:$j_2$] at (6,-2) {};
   \node[pnt,label=below:$j_1$] at (5,-2) {};
   \draw (1,0) to (5,-2);
   \draw (2,0) to (6,-2);
   \draw (3,0) to (7,-2);
   \draw (4,0) to (8,-2);
\end{tikzpicture}
\caption{A two-line diagram of a $k$-crossing}
\label{fig:2linecross}
\end{figure}

Similarly, when $k$ arcs nest in the standard representation,  these $k$ line segments mutually intersect in the two-line diagram. Figure~\ref{fig:2linenest} shows a $k$-nesting from Figure~\ref{fig:knesting} drawn in a two-line diagram.
\begin{figure}
\figurefontsize\centering
\begin{tikzpicture}
   \node[pnt,label=above:$i_1$] at (1,0) {};
   \node[pnt,label=above:$i_2$] at (2,0) {};
   \node[label=above:$\dots$] at (3,0) {};
   \node[pnt,label=above:$i_k$] at (4,0) {};
   \node[pnt,label=below:$j_k$] at (5,-2) {};
   \node[label=below:$\dots$] at (6,-2) {};
   \node[pnt,label=below:$j_2$] at (7,-2) {};
   \node[pnt,label=below:$j_1$] at (8,-2) {};
   \draw (1,0) to (8,-2);
   \draw (2,0) to (7,-2);
   \draw (3,0) to (6,-2);
   \draw (4,0) to (5,-2);	
\end{tikzpicture}
\caption{A two-line diagram of a $k$-nesting}
\label{fig:2linenest}
\end{figure}
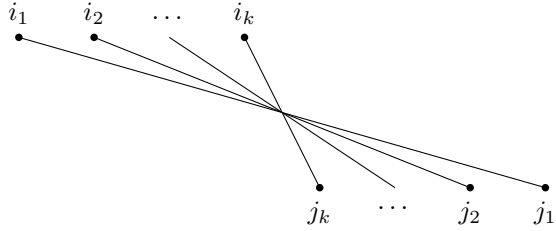
Again, note the effect of reverse relabelling the closers (lower vertices) turning a $k$-crossing to a $k$-nesting.

As one relabels $\mu$'s closer labels in reverse in the two-line diagram to produce a new matching $\nu$, each pair of nesting arcs in $\mu$ forms a crossing in $\nu$, and vice versa. To map it back, just perform the same procedure to $\nu$ to get back to $\mu$.
\end{proof}

\subsection{An Example}

We take a matching $ \mu =\{(1,10), (2,6), (3,7), (4,8), (5, 9) \}$, and label its closers in reverse, then redraw the new matching to produce the image of  $\mu$, namely $\nu$.

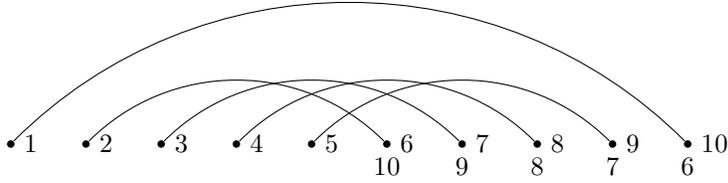
\begin{figure}
\figurefontsize\centering
\begin{tikzpicture}
   \foreach \i in {1,...,10} {
      \node[pnt] (a\i) at (\i, 0) {};
      \node[right] at (a\i.east) {$\i$};
   }
   \foreach \i [evaluate =\i as \j using int(16-\i)] in {6,...,10}
      \node[below] at (a\i.south) {$\j$};
   \draw[bend left=45] (a1) to (a10);
   \draw[bend left=45] (a2) to (a6);
   \draw[bend left=45] (a3) to (a7);
   \draw[bend left=45] (a4) to (a8);
   \draw[bend left=45] (a5) to (a9);
\end{tikzpicture}
\caption{The matching $\mu$ with new closer labels below. }
\label{fig:mu}
\end{figure}
Reversing the labels of the closers, namely the vertex set of $\{  6, 7, 8, 9, 10\}$ seen on the lower labels of Figure~\ref{fig:mu}, we get the image shown in Figure~\ref{fig:nu}. Note that $L(\mu) = (4;6,4,1)$ whereas $L(\nu)=(6,4,1;4)$.
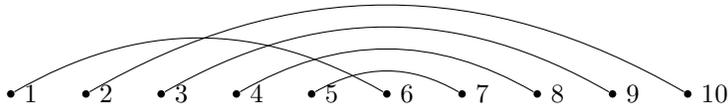
\begin{figure}
\figurefontsize\centering
\begin{tikzpicture}
   \foreach \i in {1,...,10} {
      \node[pnt] (a\i) at (\i, 0) {};
      \node[right] at (a\i.east) {$\i$};
   }
   \draw[bend left=30] (a1) to (a6);
   \draw[bend left=30] (a2) to (a10);
   \draw[bend left=30] (a3) to (a9);
   \draw[bend left=30] (a4) to (a8);
   \draw[bend left=30] (a5) to (a7);
\end{tikzpicture}
\caption{The new matching $\nu$,  the image of  $\mu$.}
\label{fig:nu}
\end{figure}

\subsection{Extension of the Subclass}

The restriction to Type OC can be lifted to indecomposable matchings with opener--closer--opener--closer configuration as in Figure~\ref{fig:extend}, called \emph{Type OCOC} in short. 

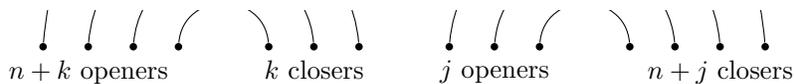
\begin{figure}
\figurefontsize\centering
\begin{tikzpicture}[scale=0.6]
   \foreach \i in {1,2,3,4,6,7,8,10,11,12,14,15,16,17}
      \node[pnt] (a\i) at (\i, 0) {};
   \node[below] at (a2.south) {$n+k$ openers};
   \node[below] at (a7.south) {$k$ closers};
   \node[below] at (a11.south) {$j$ openers};
   \node[below] at (a16.south) {$n+j$ closers};
   \clip (1,0) rectangle (17,0.8);
   \foreach \i/\j in {a1/a17, a2/a8, a3/a7, a4/a6, a10/a16, a11/a15, a12/a14}
      \draw[bend left=85, looseness=1.5] (\i) to (\j);
\end{tikzpicture}
\caption{The arc diagram of a Type OCOC indecomposable matching}
\label{fig:extend}
\end{figure}

\begin{algorithm}[Triple Reverse]
Take  a matching of Type OCOC as shown in Figure~\ref{fig:extend}. 
\begin{enumerate}[label=Step \arabic*,leftmargin=*]
\item\label{PTR1} First reverse labelling: Reverse the labels of the first consecutive block of $(n+k)$ openers by writing the new labels under the original vertex labels in the arc diagram.

\item\label{PTR2} Second reverse labelling: Reverse the labels of the second (the last) block of $(n+j)$ closers by writing the new labels under the original vertex labels in the arc diagram.
 
\item\label{PTR3} Update the situation: Draw the resulting matching using the new labels under the original vertex labels after the first two reversals.

\item\label{PTR4} Third reverse labelling: Reverse the labels of the openers (or closers) of $n$ connecting arcs, then draw the final arc diagram to complete the switching of nesting and crossing numbers.

The reason for the third reversal is that when two reversals of vertex labelling  have been done, once to the first indecomposable sub-block and once to the second indecomposable sub-block producing the switching of all crossing and nesting numbers for these two sub-blocks, the connecting arcs ($n$ of these) from one indecomposable sub-block to the other were reversed two times, resulting in the same nesting and crossing numbers for these connecting arcs (just inside out). Since the important relationship between these $n$ connecting arcs and each indecomposable sub-block is the placement of $n$  openers of the first sub-block and  $n$ closers of the last sub-block where the placement is already determined by the first two reversals of the procedure, the only step left is to reverse the labels of $n$ closers (or openers, but not both) of the connecting arcs.

\end{enumerate}
\end{algorithm}

We apply \emph{Procedure Triple Reverse} to the matching
 given in Figure~\ref{fig:extriple}.
\begin{figure}
\figurefontsize\centering
\begin{tikzpicture}[scale=0.6]
   \foreach \i in {1,2,3,4, 5,6,7,8, 9,10, 11,12,13,14,15, 16, 17, 18} {
      \node[pnt] (a\i) at (\i, 0) {};
      \node[below] at (a\i.south) {$\i$};
   }
   \foreach \i/\j in {a1/a10, a2/a8, a6/a7, a4/a9, a12/a14, a11/a15, a13/a17}
      \draw[bend left=40] (\i) to (\j);
   \foreach\i/\j in {a3/a16, a5/a18}
      \draw[bend left=45, cnt] (\i) to (\j);
\end{tikzpicture}
\caption{An example with two indecomposable matchings connected by $2$ crossing arcs}
\label{fig:extriple}
\end{figure}
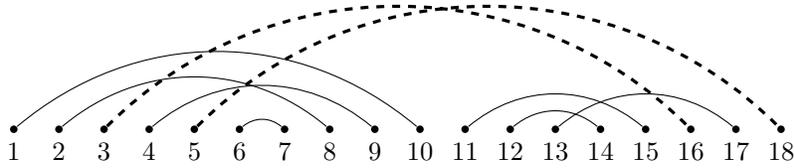
\ref{PTR1} and \ref{PTR2} of Procedure Triple Reverse  reverse the labels of  the first block of openers and the last block of closers as shown in Figure~\ref{fig:extriple1}.

\begin{figure}
\figurefontsize\centering
\begin{tikzpicture}[scale=0.65]
   \foreach \i in {1,2,3,4, 5,6,7,8, 9,10, 11,12,13,14,15, 16, 17, 18} {
      \node[pnt] (a\i) at (\i, 0) {};
      \node[right, inner sep=0, outer sep=0] at (a\i.east) {$\i$};
   }
   \foreach \i [evaluate=\i as \j using int(7-\i)] in {1,...,6}
      \node[below] at (a\i.south) {$\j$};
   \foreach \i [evaluate=\i as \j using int(32-\i)] in {14,...,18}
      \node[below] at (a\i.south) {$\j$};
   \foreach \i/\j in {a1/a10, a2/a8, a6/a7, a4/a9, a12/a14, a11/a15, a13/a17}
      \draw[bend left=40] (\i) to (\j);
   \foreach\i/\j in {a3/a16, a5/a18}
      \draw[bend left=45, cnt] (\i) to (\j);
\end{tikzpicture}
\caption{An example with two indecomposable matchings connected by $2$ crossing arcs with the first two reversals of labels done below the original labels}
\label{fig:extriple1}
\end{figure}
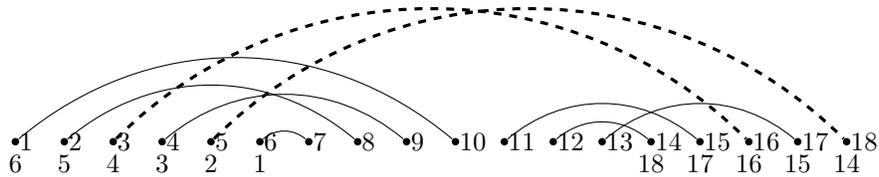

\ref{PTR3} updates by redrawing the diagram of Figure~\ref{fig:extriple1} with the new labels in increasing order from the left to the right to produce Figure~\ref{fig:extriple2}.
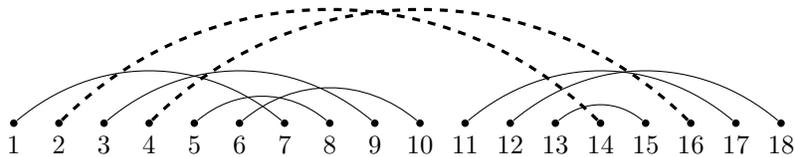
\begin{figure}
\figurefontsize\centering
\begin{tikzpicture}[scale=0.6]
   \foreach \i in {1,2,3,4, 5,6,7,8, 9,10, 11,12,13,14,15, 16, 17, 18} {
      \node[pnt] (a\i) at (\i, 0) {};
      \node[below] at (a\i.south) {$\i$};
   }
   \foreach \i/\j in {a1/a7, a5/a8, a6/a10, a3/a9, a12/a18, a11/a17, a13/a15}
      \draw[bend left=40] (\i) to (\j);
   \foreach\i/\j in {a4/a16, a2/a14}
      \draw[bend left=45, cnt] (\i) to (\j);
\end{tikzpicture}
\caption{Redrawing of the new matching after two reversals of labels}
\label{fig:extriple2}
\end{figure}
Notice that we need~\ref{PTR4},  the third reversal, for the connecting arcs to complete the bijection. In this case, we simply switch $14$ and $16$ to get the final diagram, Figure~\ref{fig:extriple3}.

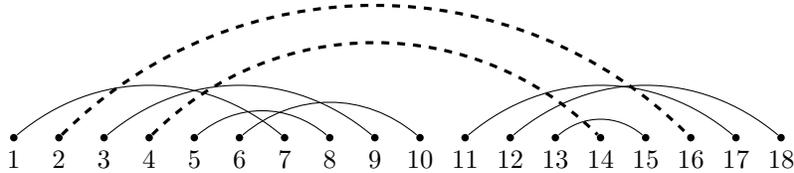
\begin{figure}
\figurefontsize\centering
\begin{tikzpicture}[scale=0.6]
   \foreach \i in {1,2,3,4, 5,6,7,8, 9,10, 11,12,13,14,15, 16, 17, 18} {
      \node[pnt] (a\i) at (\i, 0) {};
      \node[below] at (a\i.south) {$\i$};
   }
   \foreach \i/\j in {a1/a7, a5/a8, a6/a10, a3/a9, a12/a18, a11/a17, a13/a15}
      \draw[bend left=40] (\i) to (\j);
   \foreach\i/\j in {a4/a14, a2/a16}
      \draw[bend left=45, cnt] (\i) to (\j);
\end{tikzpicture}
\caption{Procedure Triple Reverse completed}
\label{fig:extriple3}
\end{figure}

The reader is encouraged to check crossing and nesting numbers for each indecomposable sub-interval (connected by solid arcs) and the entire matching (including dashed arcs) to see that all refinements of crossing and nesting numbers are switched between Figure~\ref{fig:extriple} and Figure~\ref{fig:extriple3}.

Seeing how connecting arcs over two indecomposable sub-blocks of opener--closer configuration can be treated by three label reversals, we may wish to construct recursively larger structures via more connecting arcs. Unfortunately, this tantalizing idea cannot be extended to larger structures with  broader connecting arcs because sub-indecomposable intervals entirely nested under the broad connecting arcs will remain nested under the same arcs after all label reversals, affording no opportunity for a switching of nesting and crossing numbers for these internal (relative to the broader arcs) sub-indecomposable intervals.

Let us summarize the consequence of Procedure Triple Reverse with the following theorem.

\begin{theorem}\label{thm:triple}
Given a Type OCOC matching $\mu \in \mathfrak{M}_{2n}$ such that
\[
	L(\mu) = (n_2, n_3, \dots, n_i; c_2, c_3, \dots, c_j),
\]
 Procedure Triple Reverse  is an involution mapping $\mu \in \mathfrak{M}_{2n}$ to $\nu \in \mathfrak{M}_{2n}$ where the label of $\nu$ is
\[
	L(\nu) = ( c_2, c_3, \dots, c_j;n_2, n_3, \dots, n_i).
\]
\end{theorem}

\begin{proof}
Remove connecting arcs from a given matching of Type OCOC, then what remains consists of two indecomposable Type OC matchings each of which needs one reverse labelling to switch all refinements of nesting and crossing numbers by Theorem ~\ref{thm:matching}. Keeping all original labels, Procedure Triple Reverse performs the first two reversals of labels to achieve switch of crossing and nesting statistics for each sub-indecomposable interval while holding the places for the vertices of connecting arcs. The last reverse labelling for the closers (or openers) of the connecting arcs completes the switching of all statistics.
\end{proof}

\section{Set Partitions}\label{sec:set partitions}

\subsection{Definitions and Terminology}

A standard representation of a set partition on $[n]$ consists of $n$ vertices written in a line increasingly labelled from the left and only upper arcs joining elements of the same block. For example,
Figure~\ref{fig:partition} shows the diagram of a partition of $\{1, \dots, 9\}$.

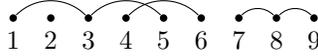
\begin{figure}
\figurefontsize\centering
\begin{tikzpicture}[scale=0.5]
   \foreach \i in {1,...,9}
        \node[pnt,label=below:$\i$] at (\i,0)(\i) {};
   \draw(1)  to [bend left=45] (3);
   \draw(3)  to [bend left=45] (5);
   \draw(4)  to [bend left=45] (6);
   \draw(7)  to [bend left=45] (8);
   \draw(8)  to [bend left=45] (9);
\end{tikzpicture}
\caption{The arc diagram representation of the partition $\{1,3,5 \}\{2\}\{4,6\}\{7,8,9\}$. }
\label{fig:partition}
\end{figure}

The vertices of a partition diagram can be of four types:
\emph{fixed points}, \emph{openers}, \emph{closers} and
\emph{transitories}. A fixed point has no incident edges, an opener
has degree one and is the left end-point of an arc, a closer has
degree one and is the right end-point of an arc, and a transitory
vertex has degree two and is the right end-point of one arc and the left
end-point of another.  In Figure~\ref{fig:partition}, the vertex
labelled $2$ is a fixed point; the vertices labelled $1$, $4$, and $7$
are openers; vertices $5$, $6$, and $9$ are closers; and $3$ and $8$
are transitory.

For normal nesting and crossing (as opposed to \emph{enhanced} nesting and crossing defined in subsection~\ref{ss:def}), fixed points can be disregarded because they do not contribute to any nesting or crossing. A transitory vertex $v$ can be redrawn as two vertices $v'$ and $v''$ such that $v'$ is a closer and $v''$ is an opener with $v' < v''$. This is a standard procedure called \emph{inflation}: splitting one vertex into two for the benefit of turning set partitions into involutions. Since fixed points do not matter, we obtain an analogous result to Theorem~\ref{thm:matching} and Procedure Triple Reverse for set partitions as well. The definitions for Type OC and Type OCOC set partitions are similarly extended from matchings.

\subsection{Result of Procedure Triple Reverse}

For Theorem~\ref{thm:partition}, since Type OC necessarily excludes transitory vertices, such set partitions are just involutions or partial matchings.

\begin{theorem}
\label{thm:partition}
Given a Type OC set partition $\pi \in \mathfrak{P}_n$, that is,  a set partition on $[n]$ with a block of consecutive openers followed by a block of consecutive closers in an indecomposable interval, with label
\[
	L(\pi) = (n_2, n_3, \dots, n_i; c_2, c_3, \dots, c_j),
\]
 there is an involution mapping $\pi \in \mathfrak{P}_n$ to a Type OC set partition, $\rho \in \mathfrak{P}_n$ where 
\[
	L(\rho) = ( c_2, c_3, \dots, c_j;n_2, n_3, \dots, n_i).
\]
\end{theorem}

\begin{proof}
Perform a single reverse labelling to the closers. This switches all refinements of nesting and crossing numbers as in the proofs for Theorem ~\ref{thm:matching}.
\end{proof}

To check whether Procedure Triple Reverse applies to a particular set partition, first perform inflation to the transitory vertices of the set partition and check if it is of Type OCOC. Note that Type OCOC forces the set partition to have at most one transitory vertex. In such a case, the transitory vertex $v$ after inflation produces $v'$ (a closer) and $v''$ (an opener) in that order. See Figure~\ref{fig:trans}. The Procedure's three relabelling steps never alter $v'$ or $v''$; thus they still remain next to each other as closer--opener for deflation. To get the set partition back, one deflates by identifying $v'$ and $v''$ to $v$. All fixed points keep their labels, otherwise  ignored for the purpose of the procedure.

\begin{figure}
\figurefontsize\centering
\begin{tikzpicture}[scale=0.6]
   \foreach \i in {1,2,3,4,6,7, 9,10,12,13,15,16,17, 18}
      \node[pnt] (a\i) at (\i, 0) {};
   \node[below] at (a2.south) {$n+k$ ops};
   \node[below] at (a7.south) {$k-1$ cls};
   \node[below] at (a9.south) {$v'$};
   \node[below] at (a10.south) {$v''$};
   \node[below] at (a13.south) {$j-1$ ops};
   \node[below] at (a17.south) {$n+j$ cls};
   \clip (1,0) rectangle (18,0.8);
   \foreach \i/\j in {a1/a18, a2/a9, a3/a7, a4/a6, a10/a17, a12/a16, a13/a15}
      \draw[bend left=85, looseness=1.5] (\i) to (\j);
\end{tikzpicture}
\caption{Inflated transitory vertex $v'$ and $v''$}
\label{fig:trans}
\end{figure}
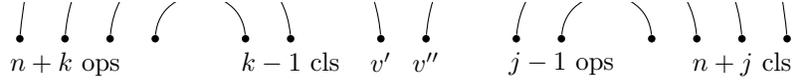

We summarize the result of Procedure Triple Reverse in the following theorem. 
\begin{theorem}
\label{thm:partitiontriple}
Given a Type OCOC set partition $\pi \in \mathfrak{P}_n$ on $[n]$ with label 
\[
	L(\pi) = (n_2, n_3, \dots, n_i; c_2, c_3, \dots, c_j),
\]
Procedure Triple Reverse is an involution mapping $\pi \in \mathfrak{P}_n$ to a Type OCOC set partition $\rho \in \mathfrak{P}_n$ where 
\[
	L(\rho) = ( c_2, c_3, \dots, c_j;n_2, n_3, \dots, n_i).
\]
\end{theorem}

\begin{proof}
Inflate transitory vertices if any. If the set partition is still Type OCOC, perform Procedure Triple Reverse then deflate the transitory if present. The result is the image of the set partition.
\end{proof}

\section{Permutations}
\subsection{Definitions and Terminology}
\label{ss:def}
The \emph{arc annotated sequence} associated with a permutation $\sigma \in \mathfrak{S}_n$ is the directed graph on the vertex set $V(\sigma) = \{ 1, 2, \dots, n\}$ with arc set
 $A(\sigma) = \{ (a, \sigma(a)) : 1 \le a \le n \}$ drawn in such a way that the edge $(a, \sigma(a))$ is an upper arc if $a \le \sigma(a)$; otherwise, $(a, \sigma(a))$ is a lower arc when $a > \sigma(a)$. This way of drawing makes the direction of each arc unique, thus the omission of arrows on arcs. Figure~\ref{fig:perm} shows a permutation with its arc diagram.
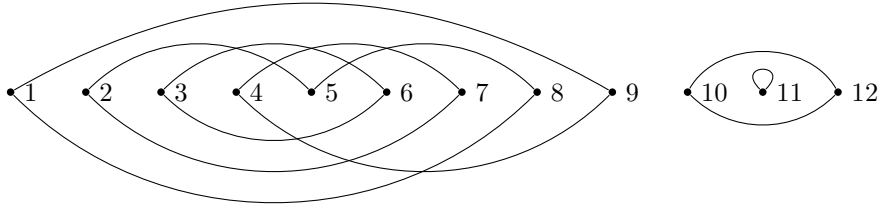
\begin{figure}
\figurefontsize\centering
\begin{tikzpicture}
   \foreach \i in {1,...,12} {
      \node[pnt] (a\i) at (\i, 0) {};
      \node[right] at (a\i.east) {$\i$};
   }
   \draw[bend left=30] (a1) to (a9);
   \draw[bend left=45] (a2) to (a5);
   \draw[bend left=45] (a3) to (a6);
   \draw[bend left=45] (a4) to (a7);
   \draw[bend left=45] (a5) to (a8);
   \draw[bend left=60] (a10) to (a12);
   \draw[loop above] (a11) to ();
   \draw[bend right=45] (a1) to (a8);
   \draw[bend right=45] (a2) to (a7);
   \draw[bend right=45] (a3) to (a6);
   \draw[bend right=45] (a4) to (a9);
   \draw[bend right=45] (a10) to (a12);
\end{tikzpicture}
\caption{The arc diagram of the permutation $\sigma=(1,9,4,7,2,5,8)(3,6)(10,12)(11)$}
\label{fig:perm}
\end{figure}

Since the arc diagram of a permutation is composed of upper and lower arc diagrams, we may inherit the terminology of set partitions for the four types of vertices and the concepts of  $k$-nesting and $k$-crossing. Figure~\ref{fig:perm} has all four types of vertices: Vertex $\{11\}$ is a fixed point; vertices $\{1, 2, 3, 4, 10\}$ are openers; $\{6, 7, 8, 9, 12\}$ are closers; and vertex $\{5\}$ is an upper transitory.
Figure~\ref{fig:permup} shows the upper arc diagram, and Figure~\ref{fig:permlower} shows the lower arc diagram of the permutation from Figure~\ref{fig:perm} drawn as upper arcs.
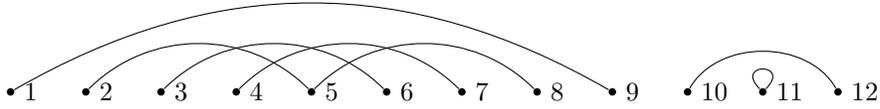
\begin{figure}
\figurefontsize\centering
\begin{tikzpicture}
   \foreach \i in {1,...,12} {
      \node[pnt] (a\i) at (\i, 0) {};
      \node[right] at (a\i.east) {$\i$};
   }
   \draw[bend left=30] (a1) to (a9);
   \draw[bend left=45] (a2) to (a5);
   \draw[bend left=45] (a3) to (a6);
   \draw[bend left=45] (a4) to (a7);
   \draw[bend left=45] (a5) to (a8);
   \draw[bend left=60] (a10) to (a12);
   \draw[loop above] (a11) to ();
\end{tikzpicture}
\caption{The upper arc diagram of the permutation $\sigma=(1,9,4,7,2,5,8)(3,6)(10,12)(11)$}
\label{fig:permup}
\end{figure}
\begin{figure}
\figurefontsize\centering
\begin{tikzpicture}
   \foreach \i in {1,...,12} {
      \node[pnt] (a\i) at (\i, 0) {};
      \node[right] at (a\i.east) {$\i$};
   }
   \draw[bend left=30] (a1) to (a8);
   \draw[bend left=30] (a2) to (a7);
   \draw[bend left=30] (a3) to (a6);
   \draw[bend left=30] (a4) to (a9);
   \draw[bend left=30] (a10) to (a12);
\end{tikzpicture}
\caption{The lower arc diagram of the permutation $\sigma=(1,9,4,7,2,5,8)(3,6)(10,12)(11)$ drawn as an arc diagram}
\label{fig:permlower}
\end{figure}
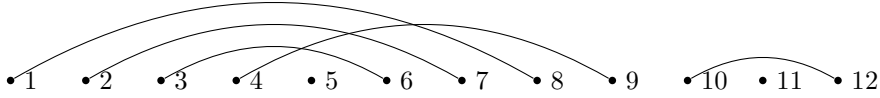

For permutations, we draw a fixed point  as an upper loop seen in Figure~\ref{fig:perm} for the singleton $\{11\}$. Thus, a fixed point contributes an opener followed by a closer.  We define an  \emph{enhanced} nesting for upper arcs as  a pair of arcs $(a, \sigma(a)), (b, \sigma(b))$ satisfying $a < b \le \sigma(b) < \sigma(a)$; normal nesting is used for lower arcs.
Similarly for upper arcs, we define an \emph{enhanced} crossing as a pair of arcs $(a, \sigma(a)), (b, \sigma(b))$ satisfying $a < b \le \sigma(a) < \sigma(b)$. Again, normal crossing is used for lower arc diagrams. The implication for an upper transitory is that inflation produces first  an opener then a closer (unlike the case with inflation of a transitory for set partitions). Doing so creates a crossing.  Figure~\ref{fig:e3nest} shows an  enhanced $3$-nesting and an enhanced $3$-crossing. In lower arc diagrams, however, fixed points have no arcs while lower transitories still inflate to a closer followed by an opener.

This slight dissymmetry to the treatment of upper and lower arcs in the definition is consistent with the literature, but is inconsequential. Corteel \cite{Corteel07} defined crossings and nestings this way because they represent better known permutation statistics like weak exceedances and some pattern avoidance. 
\begin{figure}
\figurefontsize\centering
\begin{tikzpicture}[scale=0.6]
   \node[pnt] at (0,0)(1){};
   \node[pnt] at (1,0)(2){};
   \node[pnt] at (2,0)(3){};
   \node[pnt] at (3,0)(4){};
   \node[pnt] at (4,0)(5){};
   \node[pnt] at (10,0)(1b){};
   \node[pnt] at (11,0)(2b){};
   \node[pnt] at (12,0)(3b){};
   \node[pnt] at (13,0)(4b){};
   \node[pnt] at (14,0)(5b){};
   \draw(1)  to [bend left=60] (5);
   \draw(2)  to [bend left=60] (4);
   \draw[loop above] (3) to ();
   \draw(1b)  to [bend left=45] (3b);
   \draw(2b)  to [bend left=45] (4b);
   \draw(3b)  to [bend left=45] (5b);
\end{tikzpicture}
\caption{An enhanced $3$-nesting and an enhanced $3$-crossing}
\label{fig:e3nest}
\end{figure}
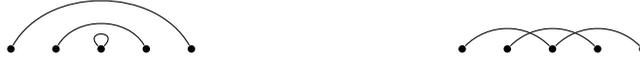

\subsection{Result of Procedure Triple Reverse}

Let $\cross(\sigma)$ be the maximal $i$ such that $\sigma$ has an $i$-crossing; similarly, let $\nest(\sigma)$ be the maximal $j$ such that $\sigma$ has a $j$-nesting. Let $S_n(i,j)$ be the number of permutations $\sigma \in \mathfrak{S}_n$ with $\cross(\sigma) = i$ and $\nest(\sigma) = j$. Then Burrill et al.\cite{BuMiPo10} proved that $S_n( i, j ) = S_n( j, i)$.

Besides maximal nesting and crossing numbers, a permutation $\sigma \in \mathfrak{S}_n$ has $nu_2$ upper $2$-nestings, $nu_3$ upper $3$-nestings, etc. and $cu_2$ upper $2$-crossings, $cu_3$ upper $3$-crossings and so on; similarly, $\sigma$ also has $nl_2$ lower $2$-nestings, $nl_3$ lower $3$-nestings, etc., $cl_2$ lower $2$-crossings and $cl_3$ lower $3$-crossings, etc. We define the \emph{label} of the permutation $\sigma$ to be
\begin{multline*}
   L(\sigma)
   = (nu_2, nu_3, \dots, nu_i;
      cu_2, cu_3, \dots, cu_j;   \\
      nl_2, nl_3, \dots, nl_k;
      cl_2, cl_3, \dots, cl_l)   ,
\end{multline*}
where $\nest(\sigma) =\max (i, k)$, and  $\cross(\sigma)=\max(j, l)$.
We similarly extend the definitions of Type OC and Type OCOC to permutations.

We are ready to state a similar result to Theorem~\ref{thm:matching} and Theorem~\ref{thm:partition} for permutations. Here we adapt the style of phrasing from \cite{Chetal07} to place emphasis on symmetric joint distribution instead of the involution.

\begin{theorem}
\label{thm:permutation}
Given a Type OC permutation $\sigma \in \mathfrak{S}_n$ with an arc diagram consisting of openers followed by closers for each indecomposable interval, the number of permutations $\sigma \in \mathfrak{S}_n$ with nesting and crossing label
\begin{multline*}
   L(\sigma)
   = (nu_2, nu_3, \dots, nu_i;
      cu_2, cu_3, \dots, cu_j; \\
      nl_2, nl_3, \dots, nl_k;
      cl_2, cl_3, \dots, cl_l)
\end{multline*}
equals the number of such  permutations with label
\begin{multline*}
   L(\sigma)
   = ( cu_2, cu_3, \dots, cu_j;
      nu_2, nu_3, \dots, nu_i;  \\
      cl_2, cl_3, \dots, cl_l;
      nl_2, nl_3, \dots, nl_k).
\end{multline*}
\end{theorem}

\begin{proof}
Use the involution that reverses the labels of all closers which may include that of a fixed point or an upper transitory. Redraw to get the image of the permutation. The involution thus implies symmetric joint distribution. Note that Type OC excludes the presence of lower transitories in the permutation.
\end{proof}

For enhanced nesting and crossing, Type OC forces the presence of at most one fixed point or upper  transitory  in the middle of an indecomposable interval because a fixed point or upper transitory is inflated to an opener--closer pair in that order.

 The main example in \cite{BuMiPo10} consists of two indecomposable intervals each satisfying the hypothesis of Theorem~\ref{thm:permutation}. We treat upper and lower diagrams separately. For the upper arc diagram of permutation $A_{\sigma}$ from Figure~\ref{fig:permup},
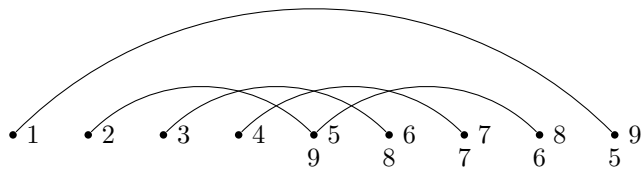
\begin{figure}
\figurefontsize\centering
\begin{tikzpicture}
   \foreach \i in {1,...,9} {
      \node[pnt] (a\i) at (\i, 0) {};
      \node[right] at (a\i.east) {$\i$};
   }
   \foreach \i [evaluate =\i as \j using int(14-\i)] in {5,...,9}
      \node[below] at (a\i.south) {$\j$};
   \draw[bend left=45] (a1) to (a9);
   \draw[bend left=45] (a2) to (a5);
   \draw[bend left=45] (a3) to (a6);
   \draw[bend left=45] (a4) to (a7);
   \draw[bend left=45] (a5) to (a8);
\end{tikzpicture}
\caption{The upper arc diagram of the first indecomposable interval  of 
$\sigma=(1,9,4,7,2,5,8)(3,6)(10,12)(11)$}
\label{fig:permup1}
\end{figure}
reverse the labels of the closers, namely the vertex set of $\{ 5, 6, 7, 8, 9\}$ seen on the lower labels of Figure~\ref{fig:permup1} to get the image shown in Figure~\ref{fig:permupimage1}.
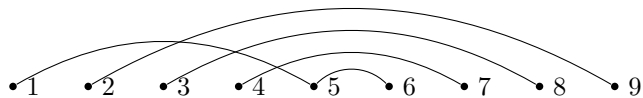
\begin{figure}
\figurefontsize\centering
\begin{tikzpicture}
   \foreach \i in {1,...,9} {
      \node[pnt] (a\i) at (\i, 0) {};
      \node[right] at (a\i.east) {$\i$};
   }
   \draw[bend left=30] (a1) to (a5);
   \draw[bend left=30] (a2) to (a9);
   \draw[bend left=30] (a3) to (a8);
   \draw[bend left=30] (a4) to (a7);
   \draw[bend left=45] (a5) to (a6);
\end{tikzpicture}
\caption{The image of  $(1,9,4,7,2,5,8)(3,6)$}
\label{fig:permupimage1}
\end{figure}
The accounting shows that Figure~\ref{fig:permup1} contains $4$ $2$-nestings and $1$ $4$-crossing, whereas Figure~\ref{fig:permupimage1} contains $4$ $2$-crossings (one of which is an enhanced crossing at $5$), and $1$ $4$-nesting. This is a different image from the one shown in \cite{BuMiPo10}.

To complete the involution, we show the result for the lower arc diagram of the first indecomposable interval of $\sigma$.
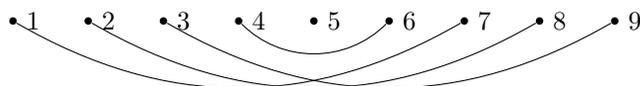
\begin{figure}
\figurefontsize\centering
\begin{tikzpicture}
   \foreach \i in {1,...,9} {
      \node[pnt] (a\i) at (\i, 0) {};
      \node[right] at (a\i.east) {$\i$};
   }
   \draw[bend right=30] (a1) to (a7);
   \draw[bend right=30] (a2) to (a8);
   \draw[bend right=30] (a3) to (a9);
   \draw[bend right=45] (a4) to (a6);
\end{tikzpicture}
\caption{The image of the lower arc diagram of the first part of $\sigma$}
\label{fig:permlowerimage1}
\end{figure}
Altogether, the permutation $\sigma$ is mapped to $\rho$ shown in Figure~\ref{fig:permimage}.
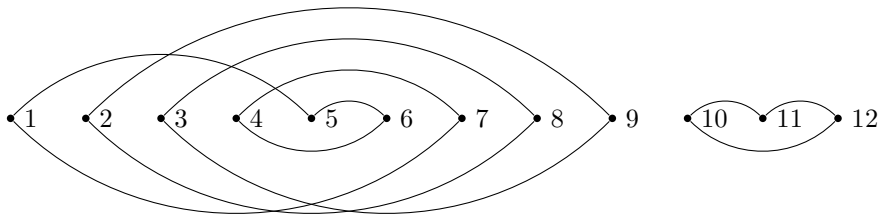
\begin{figure}
\figurefontsize\centering
\begin{tikzpicture}
   \foreach \i in {1,...,12} {
      \node[pnt] (a\i) at (\i, 0) {};
      \node[right] at (a\i.east) {$\i$};
   }
   \draw[bend left=45] (a1) to (a5);
   \draw[bend left=45] (a2) to (a9);
   \draw[bend left=45] (a3) to (a8);
   \draw[bend left=45] (a4) to (a7);
   \draw[bend left=45] (a5) to (a6);
   \draw[bend left=45] (a10) to (a11);
   \draw[bend left=45] (a11) to (a12);
   \draw[bend right=45] (a1) to (a7);
   \draw[bend right=45] (a2) to (a8);
   \draw[bend right=45] (a3) to (a9);
   \draw[bend right=45] (a4) to (a6);
   \draw[bend right=45] (a10) to (a12);
\end{tikzpicture}
\caption{The image of  $\sigma$}
\label{fig:permimage}
\end{figure}
Notice that all the refinements for crossing and nesting numbers are switched. To stitch upper and lower arc diagrams back together to get the image of the original permutation $\sigma$, simply identify vertices of the same label. This is possible because openers and closers are preserved with their original sets of labels.

In summary,  if a permutation is Type OC in each indecomposable interval, then the same procedure described for set partitions needs to be applied twice, once for the upper arcs, once for the lower arcs to find its image. One detail in stitching the result of upper arc diagram and lower arc diagram together is to keep fixed points and upper transitories as points without any arc for lower arc diagrams. Though they play no role in the reverse labelling of closers, holding their places in the diagram  aides  the putting together of both upper and lower diagrams.

Compare the image of $A_{\sigma}$ obtained from the procedure to $A_{\Psi(\sigma)}$ given by \cite{BuMiPo10} in Figure~\ref{fig:bmpimageperm} and note that refinements of crossing and nesting numbers are not interchanged in Figure~\ref{fig:bmpimageperm}.
\begin{figure}
\figurefontsize\centering
\begin{tikzpicture}
   \foreach \i in {1,...,12} {
      \node[pnt] (a\i) at (\i, 0) {};
      \node[right] at (a\i.east) {$\i$};
   }
   \draw[bend left=40] (a1) to (a8);
   \draw[bend left=40] (a2) to (a9);
   \draw[bend left=40] (a3) to (a7);
   \draw[bend left=60] (a4) to (a6);
   \draw[loop above] (a5) to ();
   \draw[bend left=45] (a10) to (a11);
   \draw[bend left=45] (a11) to (a12);
   \draw[bend right=45] (a1) to (a6);
   \draw[bend right=45] (a2) to (a7);
   \draw[bend right=45] (a3) to (a9);
   \draw[bend right=45] (a4) to (a8);
   \draw[bend right=45] (a10) to (a12);
\end{tikzpicture}
\caption{The arc diagram of $\Psi(\sigma)$ from \cite{BuMiPo10}}
\label{fig:bmpimageperm}
\end{figure}
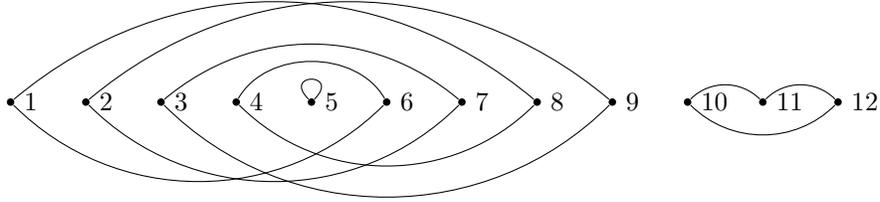

The extension from Type OC to Type OCOC works the same way for permutations as it does for matchings and set partitions. 
Note that inflation  for fixed points and transitories of upper arc diagrams is different from those of lower arc diagrams because of enhanced and non-enhanced distinction . Once applied, Procedure Triple Reverse cannot distinguish enhanced from non-enhanced statistics; thus the procedure switches all nesting and crossing numbers regardless of their types.

\section{Enumeration of Admissible Objects}
\subsection{Matchings}

The three types of indecomposable opener--closer sequences we have studied so far are:
\begin{enumerate}
\item A fixed point: no opener/closer, or an opener followed by a closer in the enhanced case after inflation. We call this \emph{Type P}.

\item A consecutive block of openers followed by a consecutive block of closers. We call this \emph{Type OC}.

\item The pattern described by Figure~\ref{fig:extend} treatable by Procedure Triple Reverse. This we call \emph{Type OCOC}.
\end{enumerate}

Procedure Triple Reverse applies to Type OCOC whereas one reverse labelling suffices for Type OC. A fixed point naturally requires no relabelling. For the sake of simplicity, we say Procedure Triple Reverse, or \emph{PTR} in short, to mean three reverse labellings or fewer.

Since these indecomposable types form building blocks of matchings we enumerate, we find their generating series first.

Let $f(x,y,z,p,s)$ denote the ordinary generating function for  partial matchings of the admissible type, that is, treatable by PTR,  where 
\begin{itemize}
\item $x$ marks closers in the first consecutive block,

\item $y$ marks openers in the second consecutive block,

\item $z$ marks connecting arcs from the first indecomposable sub-block to the second indecomposable sub-block,

\item $p$ marks  fixed points,

\item and $s$ marks the size of the partial matching.
\end{itemize}

\begin{proposition}
\label{prop:matchingGF}
\begin{enumerate}
\item\label{item:one:matchGF} Type P is generated by $ps$.

\item\label{item:two:matchGF} Type OC is generated by 
\[
O(z, p ,s)= \sum_{n \ge 1, l \ge 0} n! \binom{2n-2+l}{l} z^n p^l s^{2n+l}.
\]

\item\label{item:three:matchGF} Type OCOC is generated by 
\begin{multline*}
T(x, y ,z, p, s) = 
   \sum_{\substack{l \ge 0\\ n,j,k\ge 1}} 
      n! \binom{n+k}{k} k! \binom{n+j}{j} j! 
      \\
      \cdot
      \binom{2(n+k+j-1)+l}{l} 
      x^k y^j z^n p^l s^{2(k+j+n)+l}.
\end{multline*}
\end{enumerate}
\end{proposition}

\begin{proof}
Case~\ref{item:one:matchGF} is clear. 

Case~\ref{item:two:matchGF} has $n$ openers with $n!$ ways of closing them while $l$ fixed points are distributed between the first and last vertices, namely, $\binom{2n-2+l}{l}$ ways.

Case~\ref{item:three:matchGF} begins in Figure~\ref{fig:extend} with $n+k$ openers of which $k$ will be closed in the first block of closers, thus $\binom{n+k}{k}$ ways of choosing the $k$ openers and $k!$ ways to close them. Similarly, the second block of $j$ closers are chosen and matched to their openers of the second block. Finally, the remaining $n$ openers of the first block will close with the closers of the second block as connecting arcs in $n!$ ways. The last binomial factor is the number of ways $l$ fixed points can be distributed strictly between $2(n+k+j)$ opener/closer vertices.
\end{proof}

The coefficients of $s^n$ in $O(1, 1, s)$  grow asymptotically as
\[
(\frac{n}{e})^{n/2} (\frac{1}{\sqrt{2}})^n e^{3\sqrt{n}/2} n^{\alpha}
\]
for some constant $\alpha$ and maybe  a factor of $(\log n)^m$, but the latter would require a lot more than $100,000$ terms and more than $4$ Gb of memory. 

\begin{theorem}
\label{thm:matchingGF}
The ordinary generating function for partial matchings treatable by PTR is 
\[
f(x, y, z, p, s, ) = \frac{1}{1 - (ps + O(z,p,s) + T(x, y, z, p, s))}.
\]
\end{theorem}

\begin{proof}
We use the sequence construction.
\end{proof}

The initial numbers $1$, $1$, $2$, $4$, $10$, $26$, $76$, $232$, $756$, $2548$, $8906$, $31846$, $116422$, $432758$, $1634944$ do not match any known sequence in the On-Line Encyclopedia of Integer Sequences. However, not surprisingly, up to size $7$, the numbers are the same as the number of involutions (or self-inverse permutations). The ratio of the number of PTR treatable partial matchings to the number of involutions of size $n$ drops very quickly from $0.094223$ for $n=15$ to $0.037168$ for $n=29$ to $0.0002967$ for $n=39$. 

\subsection{Enhanced Matchings}

For enhanced partial matchings, a fixed point is drawn as a loop and contributes first an opener, then a closer; thus it can no longer be distributed at will among other vertices of an indecomposable block. Instead, it can only occur as the middle vertex of Type OC or at either one or both of the switches from an opener to a closer in Type OCOC.

\begin{proposition}
\label{prop:enhancedmatching}
\begin{enumerate}
\item Type P is generated by $ps$.

\item Type OC is generated by 
\[
O_E(z, p ,s)= (1+sp) \sum_{n \ge 1} n! z^n s^{2n}.
\]

\item Type OCOC is generated by 
\begin{multline}
T_E(x, y ,z, p, s) =    \sum_{ n,j,k\ge 1}
				 n!\binom{n+k}{k} k!\binom{n+j}{j} j!  x^k y^j z^n s^{2(k+j+n)}
				 \\
				  + 2sp \sum_{ \substack{n,j  \ge 1\\  k  \ge 0}}
				 n!\binom{n+k}{k} k!\binom{n+j}{j} j!  x^k y^j z^n s^{2(k+j+n)}
				 \\
				+(sp)^2\sum_{ \substack{n\ge 1\\  j, k  \ge 0}}
				 n!\binom{n+k}{k} k!\binom{n+j}{j} j!  x^k y^j z^n s^{2(k+j+n)}.
\end{multline}
\end{enumerate}
\end{proposition}

\begin{proof}
The proof is similar to the non-enhanced case except that fixed points are either in the middle of an OC block or not present at all, thus the factor $1 + sp$, and the lack of the binomial factor that accounts for the placement of an arbitrary number of  fixed points in the non-enhanced case. Note that the presence of a fixed point contributes an opener/closer pair in a sub-interval already, so the summation index, $k$, in the second summation, associated with that particular sub-interval begins with $0$ tracking other opener/closer pairs of that sub-interval. Similarly for two fixed points, one for each sub-interval, two indices, $j$ and $k$, in the summation of the last term start at $0$.
\end{proof}

\begin{theorem}
\label{thm:enhanced matching}
The ordinary generating function for enhanced partial matchings treatable by PTR is
\[
f_E(x, y, z, p, s, ) = \frac{1}{1 - (ps + O_E(z,p,s) + T_E(x, y, z, p, s))}.
\]
\end{theorem}

The numbers generated by this series begin with $1$, $1$, $2$, $4$, $10$, $25$,  $67$, $180$, $496$, $1370$, $3863$, $10881$, $31448$, $90280$ matching up to $n=8$ with A124500 in OEIS, otherwise a new sequence.

\subsection{Set Partitions}

Through the process of inflation of transitories, we can enumerate set partitions treatable by PTR using the enumerative results of matchings in the previous sections. However, an indecomposable interval is no longer limited to three types: Type P, OC, and OCOC because a transitory can be used to connect Types OC and OCOC arbitrarily often. 

In this case where a fixed point contributes no nesting, and a transitory is a closer followed by an opener, we can still intersperse fixed points at will, but transitories can either occur in the middle of OCOC, the switch from C to O, or join indecomposable matchings of either Type OC or OCOC.

The following proposition enumerates Type OC and OCOC for set partitions.

\begin{proposition}
\label{prop:set partitions}
\begin{enumerate}
\item\label{item:one:setpart} Type OC is generated by 
\[
   O_S(z, p ,s)
   = \sum_{n \ge 1, l \ge 0} n! \binom{2n-2+l}{l} z^n p^l s^{2n+l} 
   \qquad\text{(${}=O(z, p,s)$).}
\]

\item\label{item:two:setpart} Type OCOC is generated by 
\begin{multline*}
   T_S(x, y ,z, p, s) = 
   \sum_{\substack{l \ge 0\\ n,j,k\ge 1}} 
      n! \binom{n+k}{k} k! \binom{n+j}{j} j! 
      \\ \cdot
      \left(
         \binom{2(n+k+j-1)+l}{l} 
         + \frac1{s} \binom{2(n+k+j-1)+l-1}{l} 
      \right)
      \\ \cdot
      x^k y^j z^n p^l s^{2(k+j+n)+l}.
\end{multline*}
\end{enumerate}
\end{proposition}
\begin{proof}
Case~\ref{item:two:setpart} is Type OCOC where the switch in the middle from C to O is proper or a transitory. Both cases are accounted for in the sum of two binomial factors of the larger set of  parentheses.
\end{proof}

\begin{proposition}
\label{prop:indec set partition}
The ordinary generating function of an indecomposable set partition treatable by PTR is
\[
\begin{split}
   N_S(x, y ,z, p, s) 
   &= ps + \sum_{m \ge 1} \frac1{s^{m-1}} (O_S + T_S)^m 
   \\
   &= ps - s + \frac{s^2}{s - O_S - T_S}  .
\end{split}
\]
\end{proposition}

\begin{proof}
An indecomposable set partition is concatenated from Types OC and OCOC by identifying the  last vertex of the previous block with the first vertex of the next block,  turning it into a transitory and reducing the number of vertices by one. This accounts for the sum in the first equality. Writing the first equality using a geometric series sum formula results in the last equality.
\end{proof}

\begin{theorem}
\label{thm:setpartGF}
The ordinary generating function for set partitions treatable by PTR is 
\[
S(x, y, z, p, s, ) = \frac{1}{1 - N_S(x, y, z, p, s)}.
\]
\end{theorem}

This sequence is hard to generate because of $N_S$. The enhanced set partitions have a simpler series which is generated next.

\subsection{Enhanced Set Partitions}

In the enhanced case when a fixed point and a transitory are both inflated to an opener followed by a closer, either one can occur in the switch from opener to closer of Type OC or Type OCOC. Note that a transitory can no longer be joining an indecomposable interval to another as it did in the case of normal nesting and crossing, for doing so may  result in a set partition no longer treatable by PTR. This results in a simpler generating function.

\begin{proposition}
\label{prop:enhanced set partitions}
\begin{enumerate}
\item A fixed point is generated by $s$.

\item\label{item:one:ensetpart} Type OC is generated by 
\[
   O_{SE}(z, s)
   = (1 + s) \sum_{n \ge 1} n!\,  z^n s^{2n} + s \sum_{n \ge 1}   n!\, n z^n s^{2n}.
 \]

\item\label{item:two:ensetpart} Type OCOC is generated by 
\begin{multline*}
   T_{SE}(x, y ,z, s) = \sum_{j, k, n\ge 1}
      n! \binom{n+k}{k} k! \binom{n+j}{j} j!   x^k y^j z^n s^{2(k+j+n)}
      \\
      + 2s \sum_{\substack{k\ge 0\\ n, j\ge 1}} 
      n! \binom{n+k}{k} k! \binom{n+j}{j} j!   x^k y^j z^n s^{2(k+j+n)}
      \\
      + s^2 \sum_{\substack{j, k\ge 0\\ n\ge 1}} 
      n! \binom{n+k}{k} k! \binom{n+j}{j} j!   x^k y^j z^n s^{2(k+j+n)}
      \\
\shoveleft{      +2s \sum_{\substack{ k\ge 0\\ n, j\ge 1}} 
      n! \left( \binom{n+k}{k} k!\, k + \binom{n+k}{k+1} (k+1)! \right)}
      \\
      \shoveright{ \times  \binom{n+j}{j} j!   x^k y^j z^n s^{2(k+j+n)} }
      \\
 \shoveleft{     +2s^2 \sum_{\substack{ j, k\ge 0\\ n\ge 1}} 
      n! \left( \binom{n+k}{k} k!\, k + \binom{n+k}{k+1} (k+1)! \right) }
      \\
      \shoveright{\times \binom{n+j}{j} j!   x^k y^j z^n s^{2(k+j+n)} }
      \\
 \shoveleft{     +s^2 \sum_{\substack{j, k\ge 0\\ n\ge 1}} 
      n! \left( \binom{n+k}{k} k!\, k + \binom{n+k}{k+1} (k+1)! \right)  }
      \\
  \times  \left( \binom{n+j}{j} j!\, j + \binom{n+j}{j+1}(j+1)! \right)   x^k y^j z^n s^{2(k+j+n)}.  
\end{multline*}
\end{enumerate}
\end{proposition}

\begin{proof}
In this generating function, fixed points (and transitories) which can no longer be ignored are  marked by $s$ instead of $p$. The first term of Type OC includes a factor $(1+s)$ where $s$ accounts for a fixed point  in the middle of the switch from O to C;
the second sum is for a Type OC with a transitory for the switch from an opener to a closer. When $n$ opener/closer pairs are combined with a transitory, $n+1$ openers are available, but the transitory cannot be closed to itself otherwise turning it into a fixed point. This gives only $n$ choices for the transitory to close. Once matched to a proper closer among the $n$, the other openers have $n!$ ways to close.

The six sums from Type OCOC can be divided into two parts, the first three sums form the first half where
 \begin{enumerate*}[label=\emph{\alph*})]
  \item neither a fixed point nor a transitory is present, or
  \item only one fixed point but no transitory is present, or
  \item two fixed points are present.
  \end{enumerate*}
  The proof for the first half of $T_{SE}$ is similar to  that of enhanced matching from Proposition~\ref{prop:enhancedmatching} regarding the extra factor of $s$ or $s^2$ and the change in the start of some of the summation indices.
  
  The second half of $T_{SE}$ are the last three sums where
  \begin{enumerate*}[label=\emph{\alph*})]
  \item only one transitory but no fixed point is present, or
  \item one transitory and one fixed point are both present, or
  \item   two transitories are present.
\end{enumerate*}
  Let us consider first only one transitory but no fixed point. The transitory together with $n+k$ proper opener vertices provides $n+k+1$ openers of which $k+1$ of them must be closed with the $k$ proper closers in the first sub-interval plus the closer from the transitory. There are two ways of choosing $k+1$ openers to match these $k+1$ closers. One is to choose $k+1$ among the $n+k+1$ openers where the transitory is among them, thus $\binom{n+k}{k} k! k$ ways to close; the other is to choose $k+1$ openers among the $n+k$ proper openers to avoid the opener from the transitory, thus $\binom{n+k}{k+1} (k+1)!$ ways to close in the first sub-interval forcing the opener of the transitory to close with one of the closers from the last block of $n+j$ closers. When a transitory is present, the summation index $k$ starts at $0$. 
  
  For the last two sums, similar arguments apply for the presence of a fixed point or another transitory.
 \end{proof}

\begin{theorem}
\label{thm:enhancedsetpartGF}
The ordinary generating function for enhanced set partitions treatable by PTR is 
\[
SE(x, y, z, p,  s ) = \frac{1}{1 - (s + O_{SE} + T_{SE})}.
\]
\end{theorem}

The number of enhanced set partitions treatable by PTR begins with $1$, $1$, $2$,  $5$, $15$, $44$, $147$, $439$, $1484$, $4469$, $15217$. Up to $n=6$, the entries match A148351 in OEIS  which counts the number of three dimensional lattice walks in the first octant with some restricted step set, but otherwise unknown. Compared to Bell numbers, the ratio drops very fast also, already at $n=12$, it is $0.00379$.

\subsection{Permutations}

Traditionally, permutations have enhanced nesting and crossing for the upper arcs and non-enhanced for the lower arcs. Therefore, an indecomposable interval in the upper arcs of a permutation is the same as the enhanced set partitions: a fixed point, type OC, or type OCOC. For each type, the opener-closer sequence from the upper arcs of the permutation induces the same opener-closer sequence in the lower arcs (ignoring fixed points or upper transitories). The complication arises from the presence of lower transitories. Since they can occur arbitrarily many times, an indecomposable interval for the upper arc diagram of a permutation has many different lower arc diagrams each composed of a number of indecomposable types OC or OCOC's. This results in a complicated series which will not be derived here.

\section{PTR for Coloured Matchings, Set partitions, and Permutations}

Coloured perfect matchings were introduced by Chen and Guo~\cite{ChG11} who established symmetric joint distribution for crossing and nesting numbers where the arcs are coloured and a $k$-crossing (resp.\ $k$-nesting) applies to only $k$ arcs of the \emph{same} colour forming a $k$-crossing (resp.\ $k$-nesting). Marberg~\cite{Mar12} extended Chen and Guo's result to coloured set partitions where symmetric joint distribution still holds. A coloured set partitions is defined by Marberg as follows.

\begin{definition}[\cite{Mar12}]
Given a positive integer $r$, define an \emph{$r$-coloured partition} of $[n]$ to be a pair $(P, \phi)$ consisting of a set partition $P$ of $[n]$ and a map $\phi : \Arc(P) \rightarrow [r]$ where $\Arc(P)$ is the set of arcs of $P$ in the arc annotated diagram.

 Let $\Lambda$ be an $r$-coloured partition of $[n]$; define $\Arc(\Lambda) = \{ (i,j,t) : (i,j) \in \Arc(P), t= \phi (i,j)\}$. Note that a $P$ can have many different $\Lambda$'s, and this set, with $n$, uniquely determines $\Lambda$.
\end{definition}

Using arc annotated diagrams with the colour of each arc labelled, we can draw a $2$-coloured version of Figure~\ref{fig:partition} as in Figure~\ref{fig:colpart}.

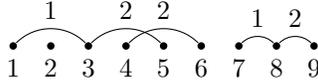
\begin{figure}
\figurefontsize\centering
\begin{tikzpicture}[scale=0.5]
   \foreach \i in {1,...,9}
        \node[pnt,label=below:$\i$] at (\i,0)(\i) {};
   \draw(1)  to [bend left=45] node[above]{$1$}  (3);
   \draw(3)  to [bend left=45] node[above]{$2$} (5);
   \draw(4)  to [bend left=45] node[above]{$2$} (6);
   \draw(7)  to [bend left=45] node[above]{$1$} (8);
   \draw(8)  to [bend left=45] node[above]{$2$} (9);
\end{tikzpicture}
\caption{A $2$-coloured partition $\Lambda =\{(1, 3, 1), (3, 5,2), (4, 6, 2), (7,8,1), (8, 9, 2) \}$. }
\label{fig:colpart}
\end{figure}

In the spirit of~\cite{Chetal07}, Marberg used vacillating $r$-partite tableaux for his bijection between the set of $r$-coloured partitions of $[n]$ and the set of vacillating $r$-partite tableaux of length $2n$, translating maximum crossing (nesting) number of a given $r$-coloured partition to the maximum number of columns (rows) in the set of vacillating $r$-partite tableaux. To complete the bijection for the switching of maximal crossing and nesting numbers, one simply takes component-wise transpose of the new vacillating $r$-partite tableaux to find its matching $r$-coloured partition.

Without passing through a sequence of $r$-partite tableaux, Procedure Triple Reverse still applies to coloured matchings and set partitions because each coloured arc assigns the end points of the arc the same colour and reverse labelling of vertices simply retains the vertex label with its colour. This modification of PTR applies to both enhanced and non-enhanced crossings and nestings.

For coloured permutations, we propose two different definitions:
\begin{definition}
\label{def:perm1}
An $r$-coloured permutation of $[n]$ is a permutation $\sigma_n$ whose arcs in its arc annotated diagram are partitioned into $r$ classes (or arcs are coloured by $r$ different colours, one colour for each arc).
\end{definition}

We remark that for coloured matchings and partitions, restricting the arcs to a subset of the colours still results in a coloured matching or partitions. However, Definition~\ref{def:perm1} may not produce a coloured permutation when the arcs are restricted to a subset of the colours, that is, without the colours, the arc annotated diagram with the restricted set of arcs is not necessarily a permutation. In order to keep substructures of a permutation according to restricted arc sets, we need the following definition:

\begin{definition}
\label{def:perm2}
We require all of Definition~\ref{def:perm1} such that arcs of the same colour also form a permutation.
\end{definition}

For coloured permutations of Definition~\ref{def:perm1}, PTR also applies analogously to the extension from set partitions to permutations for PTR in the uncoloured case. For those of Definition~\ref{def:perm2}, a restriction to Type OC or OCOC can be lifted from the entire permutation to each colour class (or factor(s) of $\sigma_n$) satisfying Type OC or OCOC.

An important difference between uncoloured and coloured cases for the application of PTR is the possibility of extension to a larger class. Once arcs are coloured, opener closer sequence can be partitioned according to colour, thus allowing a larger class to be treated by PTR.

\section{Limitation of Procedure Triple Reverse}

Our Procedure Triple Reverse implies that substructures like indecomposable intervals after the removal of dashed connecting arcs remain indecomposable intervals with the same opener/closer sequence as the start. With this assumption of keeping sub-indecomposable intervals intact, when a connecting arc envelopes three or more indecomposable intervals under it, the procedure cannot be extended any more for the simple reason that the top diagram of  Figure~\ref{fig:badex3} does not have an arc crossing all three indecomposable intervals in its image set partition. Using the bijection of Chen et. al.~\cite{Chetal07} and Krattenthaler~\cite{Kratt06}, the indecomposable substructures under the big enveloping arc disappears, giving way to one indecomposable structure. In this particular example, $3$ $2$-nestings are changed to $3$ $2$-crossings, but not with the same arc.
\begin{figure}
\figurefontsize\centering
\begin{tikzpicture}
   \foreach \i in {1,2,3,4, 5,6,7,8} {
      \node[pnt] (a\i) at (\i, 0) {};
      \node[below] at (a\i.south) {$\i$};
   }
   \foreach \i/\j in {a1/a8, a2/a3, a4/a5, a6/a7}
      \draw[bend left=30] (\i) to (\j);
\end{tikzpicture}

\quad

\begin{tikzpicture}
    \foreach \i in {1,2,3,4, 5,6,7,8} {
      \node[pnt] (a\i) at (\i, 0) {};
      \node[below] at (a\i.south) {$\i$};
   }
   \foreach \i/\j in {a1/a3, a2/a5, a4/a7, a6/a8}
      \draw[bend left=30] (\i) to (\j);   
      
 \end{tikzpicture}
 \caption{A bijective example under Chen et. al. and Krattenthaler}  
 \label{fig:badex3}
  \end{figure}
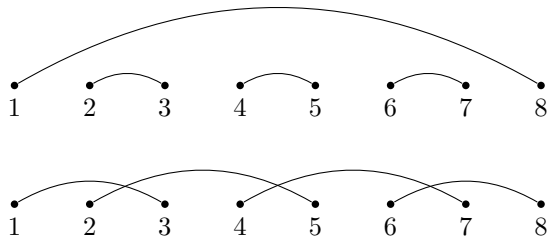

The condition of keeping sub-indecomposable intervals intact may seem restrictive; however, we gain the simplicity of staying within arc diagrams for the bijection of switching crossing and nesting numbers with the added bonus that all refinements of such numbers are also switched.

\section{Further investigation}

The  algebraic nature of this involution leads one to explore matrix representation of this subclass of permutations and their resulting matrices after the involution. The way an anti-diagonal sequence of $1$'s (called North-East chains in the language of growth diagrams in~\cite{Kratt06}) switches to a main diagonal sequence of $1$'s (South-East chains) in each block matrices of the permutation matrix and how the block matrices themselves also switch anti-diagonal/main diagonal directions leads one to think that a bijective map via matrices similar to de Mier's fillings and Krattenthaler's growth diagrams~\cite{deMi07, Kratt06} may be found for permutations.

The inspiration to the second proof from M\'{e}ndez and Rodriguez~\cite{MeRo08} using two-line representation for permutations can be applied to $q$-counting because each crossing represents an inversion. Enumeration of crossing and nesting numbers is a difficult problem; however, translating two-line representations of  such statistics to $q$-counting may shed some light on the log concavity of the sequence of numbers like $k$-noncrossing partitions or permutations as conjectured by Burrill, Elizalde, Mishna, and Yen in~\cite{BuElMiYe11}.

\section{Acknowledgments}
The author would like to thank Marni Mishna, Sophie Burrill, and Brad Jones for helpful discussions, and Manuel Kauers for finding asymptotic behaviour of the coefficients of $O(1, 1, s)$ for a given size of Type OC partial matchings using the machinery developed in the Research Institute of Symbolic Computation in Johannes Kepler University in Austria.

\bibliographystyle{elsarticle-num}

\end{document}